\documentclass[12pt]{article}
\usepackage[utf8]{inputenc}
\usepackage[T1]{fontenc}
\usepackage{mathpazo}

\usepackage{geometry}
\usepackage{amsmath}
\usepackage{amssymb}
\usepackage{amsthm}
\usepackage{mathtools}
\usepackage{braket}
\usepackage{commath}
\usepackage{tikz-cd}
\usepackage{booktabs}
\usepackage{afterpage}

\usepackage{hyperref}
\hypersetup{colorlinks = true}

\newcommand*{\vcenteredhbox}[1]{\begingroup
	\setbox0=\hbox{#1}\parbox{\wd0}{\box0}\endgroup}

\usepackage{comment}

\usepackage[style=alphabetic,useprefix,hyperref,backend=bibtex,maxcitenames=99,maxbibnames=99,maxalphanames=99]{biblatex}
\theoremstyle{plain}
\newtheorem{proposition}{Proposition}[section]
\newtheorem{lemma}[proposition]{Lemma}
\newtheorem{corollary}[proposition]{Corollary}
\newtheorem{theorem}[proposition]{Theorem}

\theoremstyle{definition}
\newtheorem{definition}[proposition]{Definition}
\newtheorem{remark}[proposition]{Remark}

\newcommand{\C}{\mathbb{C}}
\newcommand{\Z}{\mathbb{Z}}

\newcommand{\Q}{\mathbb{Q}}
\newcommand{\R}{\mathbb{R}}
\newcommand{\PP}{\mathbb{P}}
\newcommand{\NS}{\mathrm{NS}}

\newcommand{\QQ}{\mathcal{Q}}
\newcommand{\RR}{\mathcal{R}}
\newcommand{\U}{\mathcal{U}}
\newcommand{\XX}{\mathcal{X}}

\usepackage{titlesec}

\titleformat*{\section}{\large\scshape}
\titleformat*{\subsection}{\scshape}
\titleformat*{\subsubsection}{\scshape}

\title{\normalsize{\textbf{COMPLEX MULTIPLICATION AND NOETHER-LEFSCHETZ LOCI OF\\THE TWISTOR SPACE OF A K3 SURFACE}}}
\author{\normalsize{\textsc{Francesco Viganò}}\footnote{LSGNT (London School of Geometry and Number Theory), Centre for Doctoral Training across Imperial College London, University College London and King's College London, \href{mailto:francesco.vigano.20@ucl.ac.uk}{francesco.vigano.20@ucl.ac.uk} .}}
\date{}

\begin{document}

\sloppy

\maketitle

\begin{abstract}
	\small{For an algebraic K3 surface with complex multiplication (CM), algebraic fibres of the associated twistor space away from the equator are again of CM type. In this paper, we show that algebraic fibres corresponding to points at the same altitude of the twistor base $S^2\simeq \PP^1_\C$ share the same CM endomorphism field. Moreover, we determine all the admissible Picard numbers of the twistor fibres.}
\end{abstract}

A projective (or equivalently, algebraic) complex K3 surface $X$ is said to be CM (complex multiplication) if the endomorphism field $K_{T(X)}$ of the Hodge structure on the transcendental lattice $T(X) = \NS(X)_\Q^\perp \subseteq H^2(X,\Q)$ is a CM field, and $\dim_{K_{T(X)}} T(X) =1$. Denote by $\XX \to \PP^1_\C$ the twistor space of the projective K3 surface $X$ associated with a K\"ahler class given by an ample class $\ell = c_1(L) \in H^2(X,\Q)$ (see Section \ref{section.shortdescription} for a short description). Despite the transcendental nature of the twistor construction, the fibres $\XX_\zeta$ that are again algebraic share some arithmetic properties. In particular, all algebraic fibres away from the equator $\PP^1_\C \simeq S^2$ are CM, and the corresponding CM endomorphism fields share the same totally real maximal subextension \cite[Theorem 5.3]{Huy}.

\medskip

In this paper we prove that the CM fields corresponding to fibres at the same altitude of $S^2\simeq \PP^1_\C$ coincide (Theorem \ref{theorem.sameCM}). This result will be proven by introducing an action of the topological multiplicative group $K_{T(X)}^\times$ on the Noether-Lefschetz locus of the upper-half sphere.

\begin{theorem}\label{theorem.0.1}
	Consider the twistor space $\XX \to \PP^1_\C$ associated with a projective complex K3 surface $X$ with complex multiplication. Assume that $\zeta_1, \zeta_2 \in \PP^1_\C$ are two points of Picard jump at the same altitude and not on the equator. Then $\XX_{\zeta_1}$ is algebraic if and only if $\XX_{\zeta_2}$ is such. If so, then the CM endomorphism fields of these K3 surfaces coincide. Moreover, the set of points of Picard jump at the same altitude of $\zeta_1$ (and $\zeta_2$) is countable and dense in the circle at that altitude.
\end{theorem}

\begin{center}
	\includegraphics[width=0.3\textwidth]{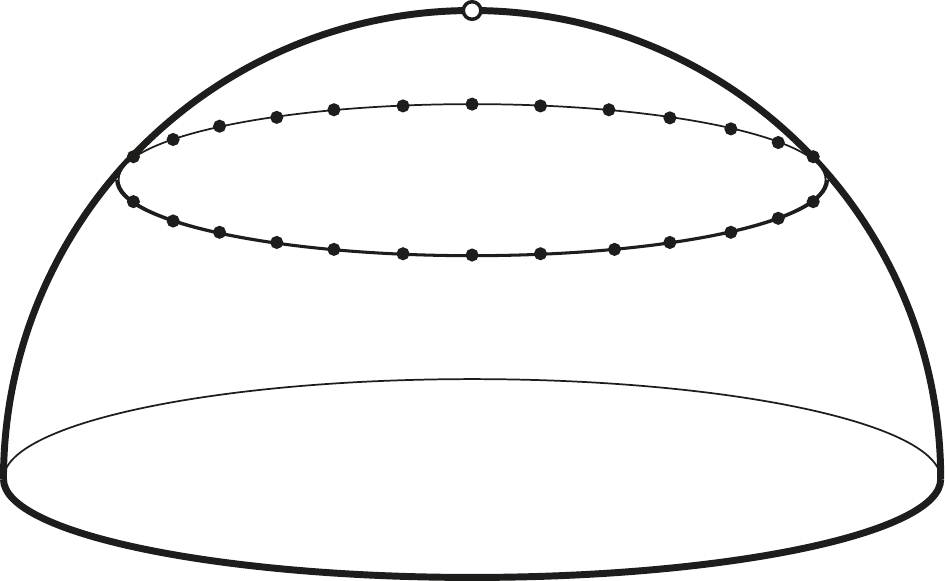}
\end{center}

\medskip

The Picard number of a K3 surface is the rank of its Néron-Severi group. Huybrechts proved that, if a fibre $\XX_\zeta$ has excessive Picard number (that is, bigger than the original Picard number $\rho(X)$), then $\zeta$ lies on the equator of $S^2 \simeq \PP^1_\C$ \cite[Proposition 3.2]{Huy}. We prove, in the CM case, that there is only one admissible excessive Picard value (Theorem \ref{theorem.geometricpicardjump}).

\begin{theorem}\label{theorem.0.2}
	Consider the twistor space $\XX \to \PP^1_\C$ associated with a projective complex K3 surface $X$ with complex multiplication. If $\zeta$ is a point of Picard jump on the equator, then
	\[
	\rho(\XX_\zeta) =  10 + \frac{\rho(X)}{2}.
	\]
	Moreover, the Noether-Lefschetz locus of the equator is dense in the equator. More precisely, the locus of points on the equator whose fibres are algebraic K3 surfaces is dense in the equator.
\end{theorem}

\begin{center}
	\vcenteredhbox{\begin{tabular}{cc}
			\toprule
			&
			\begin{tabular}{c}
				Admissible values \\
				of $\rho(\XX_\zeta)$, CM case \\
			\end{tabular}
			\\
			\toprule
			\begin{tabular}{c}
				Outside the \\
				equator \\
			\end{tabular}
			& $\rho(X)-1, \rho(X)$ \\
			\midrule
			\begin{tabular}{c}
				On the \\
				equator \\
			\end{tabular}
			& $\rho(X)-1, 10 + \frac{\rho(X)}{2}$ \\
			\bottomrule
	\end{tabular}}
	\qquad \qquad
	\vcenteredhbox{\includegraphics[width=0.3\textwidth]{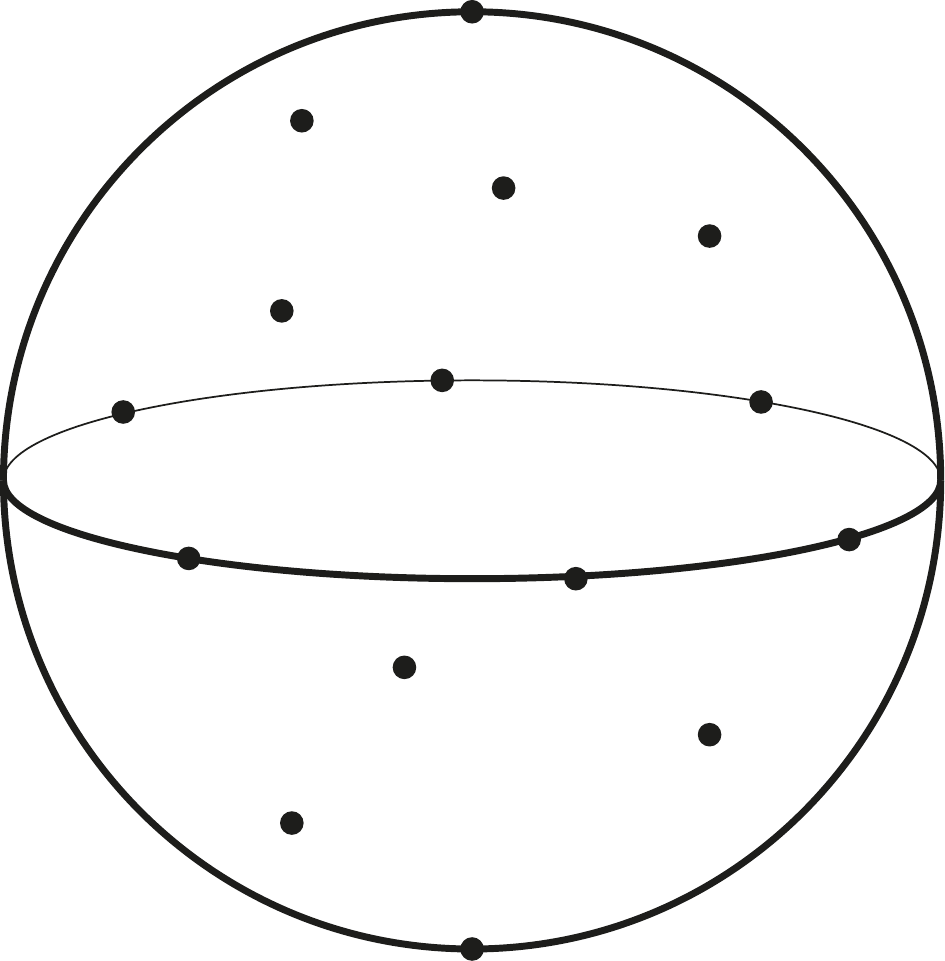}}
\end{center}

\bigskip
\bigskip

\textbf{Outline.} Section \ref{section.preliminaries} contains some basic information about the geometry of the twistor space, Hodge structures of K3 type and CM fields. In Section \ref{section.twistorsphere} we describe a family of Hodge structures, parameterized by a sphere, attached to a given Hodge structure of K3 type. This construction is, at the level of Hodge structures, the algebraic equivalent of the twistor space. Afterwards, in Section \ref{section.excessive}, we prove Proposition \ref{proposition.picardjumpeq} (algebro-equivalent of Theorem \ref{theorem.0.2}) and other results characterizing points of Picard jump on the equator of the mentioned sphere. Two actions of the topological multiplicative group $K_{T(X)}^\times$ are introduced in Section \ref{section.actions}; these will play a key role in the proof of Corollary \ref{corollary.action1} (algebro-equivalent of Theorem \ref{theorem.0.1}) and related statements. Finally, the translation into geometric terms is exposed in Section \ref{section.geometry}, and Theorems \ref{theorem.0.1} and \ref{theorem.0.2} are proven.

\bigskip

\textbf{Acknowledgements.} I thank François Charles for introducing me to the subject, and both him and Daniel Huybrechts for our prolific discussions and their significant comments. I am also thankful to Paolo Stellari for his useful remarks in the revision of this paper.

\pagebreak

\section{Preliminaries}\label{section.preliminaries}

We present a short description of the twistor space of a complex K3 surface, and we recall some facts about Hodge structures of K3 type and CM fields.

\subsection{The twistor space of a complex K3 surface}\label{section.shortdescription}
A complex K3 surface $X$ admits a sphere of different complex structures. More explicitly, one can attach a set of different complex structures $I_\zeta$ to the underlying differentiable manifold $X$, so that $(X,I_\zeta)$ is again a K3 surface, and these structures are parameterized by a sphere $\zeta \in S^2 \simeq \PP^1_\C$. These K3 surfaces can be patched together into a $3$-dimensional complex manifold $\XX$, called \emph{twistor space} (see \cite[Section 3.F]{HKLR}, \cite{Hit} or \cite[Chapter 7]{Joyce} for more on the twistor space and the details of its construction). $\XX$ comes with a holomorphic map $\XX \to \PP^1_\C$ with the property that, for $\zeta \in \PP^1_\C \simeq S^2$, the fibre $\XX_\zeta$ is the K3 surface $(X,I_\zeta)$. In fact, this construction is non-canonical, and depends on the choice of a K\"ahler class of $X$ (any K3 surface is K\"ahler, \cite{Siu}). The original K3 surface corresponds to one of the poles of $\PP^1_\C \simeq S^2$.

\medskip

Even if the original K3 surface $X$ is algebraic (or equivalently, projective), $\XX$ is no longer algebraic, and has to be thought of as of transcendental nature. The twistor space of a K3 surface plays an important role in several situations. Among all, it provides an example of K3-fibration over a compact base, and it is used in a modern proof of Torelli Theorem (any two K3 surfaces are connected by a finite path of twistor lines, see \cite[Section 7.3]{HuyK3}).

\subsection{Hodge structures of K3 type}

By \emph{polarized Hodge structure of K3 type} we mean the data of a vector space $T$ over $\Q$ of dimension $r\ge 2$, endowed with a symmetric bilinear form $( \ . \ )$ of signature $(2,r-2)$, and a decomposition
\[
T_\C = T^{2,0} \oplus T^{1,1} \oplus T^{0,2}
\]
such that the $\C$-linear extension of $( \ . \ )$ satisfies:
\begin{enumerate}
	\item the subspaces $T^{1,1}$ and $T^{2,0} \oplus T^{0,2}$ are orthogonal;
	\item $( \ . \ )$ is positive definite on $P_T = (T^{2,0} \oplus T^{0,2}) \cap T_\R$ and $T^{2,0}, T^{0,2} \subseteq T_\C$ are isotropic;
	\item complex conjugation on $T_\C$ preserves $T^{1,1}$ and exchanges $T^{2,0}$ and $T^{0,2}$;
	\item $\dim_\C T^{2,0} =1$.
\end{enumerate}

We will denote by $\sigma$ a $\C$-generator of $T^{2,0}$. Note that $\overline{\sigma}$ generates $T^{0,2}$, and the required conditions give $(\sigma.\sigma)=0$, $(\sigma.\overline{\sigma})>0$, $(\Re(\sigma))^2=(\Im(\sigma))^2$ and $(\Re(\sigma).\Im (\sigma))=0$. The plane $P_T$ is considered with the orientation given by the basis $\{\Re(\sigma), \Im(\sigma)\}$. We will also assume that $T$ is irreducible.

\medskip

Define $K_T$ to be the ring of endomorphism of the Hodge structure $T$. As $T$ is irreducible, $K_T$ is a division algebra. It was pointed out by Zarhin \cite{Zar} that $K_T$ is indeed a number field, endowed with an embedding $K_T \hookrightarrow \C$. This embedding is defined in the following way: as $\varphi$ preserves the $(2,0)$-part of $T$, $\varphi(\sigma)= \varphi \cdot \sigma$, where we identify $\varphi$ with a scalar in $\C$. Zarhin proved also that $K_T$ is either totally real or complex multiplication (CM). Notice that $T$ is naturally a $K_T$-module and, therefore, it becomes a $K_T$-vector space.

\begin{definition}
	$T$ is said to be:
	\begin{itemize}
		\item \emph{of totally real type} if $K_T$ is totally real,
		\item \emph{of almost complex multiplication type}, or \emph{almost CM}, if $K_T$ is CM, and
		\item \emph{of complex multiplication type}, or \emph{CM}, if $K_T$ is CM and $\dim_{K_T} T =1$.
	\end{itemize}
\end{definition}

The last condition seems technical, but turns out to be extremely useful in several situations. In fact, one can define the CM type case by requiring the sole condition $\dim_{K_T} T =1$ to be satisfied. Indeed, van Geemen \cite[Lemma 3.2]{vGe} proved that, whenever $K_T$ is totally real, $\dim_{K_T} T \ge 3$ holds. Of course, the condition $\dim_{K_T} T =1$ is equivalent to $[K_T \colon \Q]=r$.

\medskip

As the form $( \ . \ )$ is non-degenerate, we can define the \emph{transpose} $\varphi'$ of $\varphi$ by the condition
\[
(\gamma.\varphi(\delta)) = (\varphi'(\gamma).\delta),
\]
for any $\gamma, \delta \in T$. $\varphi'$ is in fact an element of $K_T$, and corresponds to the complex conjugate of $\varphi$ via the embedding $K_T \hookrightarrow \C$, that is: $\varphi'(\sigma)=\overline{\varphi} \cdot \sigma$ (see \cite[Remark 2.6]{Huy} or \cite[Chapter 3]{HuyK3}). In particular, an element $\varphi \in K_T$ is an isometry for $( \ . \ )$ if and only if its image in $\C$ has unitary norm. We will denote by $K^0_T=K_T \cap \R$ the real part of $K_T$, once seen $K_T$ as a subfield of $\C$ via the prescribed embedding. Notice that $K_T^0$ can be characterized by the subfield of $K_T$ of self-transpose endomorphisms.

\medskip

The \emph{period field} $k_T$ of $T$ is defined in the following way. Let $\gamma$ vary in $T$. Among the periods $(\sigma.\gamma)$, at least one is not zero, say for $\tilde{\gamma}$, since $\sigma \neq 0$ and $( \ . \ )$ is non-degenerate.\footnote{More precisely, if $T$ is irreducible, $(\sigma.\gamma)=0$ if and only if $\gamma=0$.} The period field is defined to be the subfield of $\C$ generated over $\Q$ by the quotients $(\sigma.\gamma)/(\sigma.\tilde{\gamma})$. Clearly, it is enough to consider these quotients only for $\gamma$ varying in a $\Q$-basis of $T$. For a proof of the following result, see \cite[Lemma 2.10]{Huy}.

\begin{lemma}\label{lemma.CMeasy}
	Assume that $T$ is a polarized irreducible Hodge structure of K3 type with complex multiplication (i.e.\ $K_T$ is a CM field and $\dim_\Q T =[K_T : \Q]$). Then
	\begin{enumerate}
		\item The endomorphism field $K_T$ and the period field $k_T$ coincide (as subfields of $\C$);
		\item For any basis $\{\gamma_i\}$ of $T$ and any $\sigma \in T^{2,0}$, $\sigma \neq 0$, the coordinates $x_i = (\sigma . \gamma_i)$ satisfy
		\[
		K_T=k_T=\bigoplus_{i=1}^r \Q \cdot (x_i/x_1).
		\]
	\end{enumerate}
\end{lemma}

\subsection{Useful facts on CM fields}

First of all, we recall a result proven by Blanksby and Loxton \cite{BL}.

\begin{theorem}
	If $E\subseteq \C$ is a CM field, then $E=\Q(\alpha)$, for a primitive element $\alpha$ satisfying $\abs{\alpha}=1$. 
\end{theorem}

We now present other elementary properties carried by CM fields.

\begin{lemma}\label{lemma.r/2}
	Let $E \subseteq \C$ be a number field given as a subfield of $\C$, and let $e \in E$, $e \neq 0$. Then the dimension of the $\Q$-vector space $(\R \cdot e) \cap E$ does not depend on $e$, and is equal to $\dim_\Q (E \cap \R)$.
\end{lemma}

\begin{proof}
	It is enough to notice that $(\R \cdot e) \cap E = e \cdot (E \cap \R)$.
\end{proof}

This Lemma assumes a particular form in the case of a CM field.

\begin{corollary}\label{corollary.imaginary}
	Let $E \subseteq \C$ be a CM field given as a subfield of $\C$, and let $\alpha$ be a generator of $E$ satisfying $\abs{\alpha}=1$, $n = [E : \Q]$. For any $e \in E$, $e \neq 0$, the dimension of the $\Q$-vector space $(\R \cdot e) \cap E$ does not depend on $e$, and is equal to $n/2$.
\end{corollary}

\begin{proof}
	The statement follows from Lemma \ref{lemma.r/2}, together with the fact that $E \cap \R$ is the maximal totally real subfield of $E$, satisfying $\dim_\Q (E \cap \R)=n/2$.
\end{proof}

\begin{proposition}\label{proposition.density}
	Assume that a number field $E$ is given as a subfield of $\C$. Suppose that $E \nsubseteq \R$ and $\overline{E}=E$. Then $E \cap S^1$ is dense in $S^1$.
\end{proposition}

\begin{proof}
	It is enough to show that there exist elements $\alpha \in E$ lying on the circle of arbitrary small non-zero argument (density follows taking powers of these elements). Let $\delta$ be an element of $E \setminus \R$. For $r \in \Q$, define
	\[
	\alpha_r = \frac{\delta + r}{\overline{\delta}+r} \in E.
	\]
	Then $\abs{\alpha_r}=1$ and $\alpha_r$ has arbitrary small non-zero argument, since $\lim_{r \to \infty} \alpha_r = 1$.
\end{proof}

\begin{corollary}\label{corollary.densityCM}
	If $E$ is a CM field given as a subfield of $\C$, then $E \cap S^1$ is dense in $S^1$.
\end{corollary}

\begin{proof}
	A CM field satisfies the hypotheses of Proposition \ref{proposition.density}.
\end{proof}

\section{Twistor sphere of Hodge structures}\label{section.twistorsphere}

Associated with $T$ and an abstract class $\ell$ of positive square, there exists a sphere of related Hodge structures. Here we outline its construction, following \cite[Section 3]{Huy}.

\medskip

If one wishes to keep in mind the geometric picture, they should think of $T$ as the transcendental lattice of a projective complex K3 surface $X$, of $( \ . \ )$ as the restriction of the cup product on $H^2(X,\Q)$ to $T$, and of $\ell$ as an ample class of $X$. Altered Hodge structures (or sub-Hodge structures of these) correspond to transcendental lattices of the K3 surfaces constituting the twistor space of $X$ (compare with Remark \ref{remark.differentpicardnumbers}).

\medskip

Fix a positive integer $d \in \Z_{>0}$. We extend $T$ to the Hodge structure of K3 type $T\oplus \Q \ell$, of dimension $r+1$ and endowed with a form $( \ . \ )$ of signature $(3,r+1-3)$, by declaring $\ell$ to be of type $(1,1)$, orthogonal to $T$, and to satisfy $(\ell.\ell)=d$. $P_T \oplus \R \ell \subseteq T_\R \oplus \R \ell$ is, therefore, a positive $3$-space. Notice that $T \oplus \Q \ell$ is no longer irreducible. The \emph{associated twistor base} is the conic
\[
\PP^1_\ell = \Set{ z=[\sigma'] \in \PP(T^{2,0}\oplus T^{0,2} \oplus \C \ell) | \big( \sigma'.\sigma' \big)=0}.
\]
Points of the conic $\PP^1_\ell$ define different Hodge structures on $T \oplus \Q \ell$, in the following sense. Any $z \in \PP^1_\ell$ defines a Hodge structure of K3 type on $T \oplus \Q \ell$, that we say corresponding to $z$: its $(2,0)$-part is the line corresponding to $z=[\sigma']$, i.e.\ the line $\C \sigma'$, its complex conjugate the $(0,2)$-part, and the $(1,1)$-part is given as the orthogonal complement of the former two. Notice that $( \ . \ )$ is positive definite on $(\C \sigma' \oplus \C \overline{\sigma}') \cap (T_\R \oplus \R \ell) = \R \Re(\sigma') \oplus \R \Im (\sigma')$; this follows immediately from the fact that $( \ . \ )$ is positive definite on $P_T \oplus \R \ell$, which contains $\R \Re(\sigma') \oplus \R \Im (\sigma')$.

\medskip

Mapping $z \in \PP^1_\ell$ to the oriented, positive real plane
\[
P_z = \langle \Re (z), \Im (z) \rangle_\R
\]
yields an identification $\PP^1_\ell \simeq \mathrm{Gr}^\mathrm{po} (P_T \oplus \R \ell)$ with the Grassmannian of oriented, positive planes in $P_T \oplus \R \ell$. The \emph{complex conjugate} $\overline{z}$, i.e.\ the point $\big[\, \overline{\sigma'}\, \big]$, corresponds to the same plane with reversed orientation:
\[
P_{\overline{z}} = \langle \Im (z), \Re (z) \rangle_\R.
\]
Indeed, the basis given by $\{\Re (z), -\Im (z)\}$ induces the same orientation as the basis given by $\{\Im (z), \Re(z)\}$.

\medskip

We define the \emph{period point} of $T$ as $x =[\sigma] \in \PP(T_\C)$. Via the natural inclusion $\PP(T_\C)\subseteq \PP(T_\C \oplus \C \ell)$, we see that both $x=[\sigma]$ and its complex conjugate $\overline{x}=[\overline{\sigma}]$ belong to the conic $\PP^1_\ell$, as $(\sigma.\sigma)=0$ and $(\overline{\sigma}.\overline{\sigma})=0$. Thinking of $P_z$ with its orientation being given as the orthogonal complement of a generator $\alpha_z$ of the line $P_z^\perp \subseteq P_T \oplus \R \ell$ provides a natural identification
\[
\PP^1_\ell \simeq \mathrm{Gr}_2^\mathrm{po} (P_T \oplus \R \ell) \simeq S^2_\ell = \Set{ \alpha \in P_T \oplus \R \ell | (\alpha.\alpha)=1}.
\]
With this identification, $x$ and $\overline{x}$ correspond to the normalizations of $\ell$ and $-\ell$. We think of them as the north and south poles of $S^2_\ell$. The \emph{equator} of the twistor base is the circle
\[
S^1_\ell = \Set{ z \in \PP^1_\ell | \ell \in P_z } \simeq \Set{ \alpha \in S^2_\ell | (\alpha.\ell)=0}.
\]
Indeed, if $\ell \in P_z$ then the line orthogonal to $P_z$ in $P_T \oplus \R \ell$ is orthogonal to $\ell$, and vice versa.

\medskip

If, for $z \in \PP^1_\ell$, we write $z=[\sigma'=a\sigma+b\overline{\sigma}+c\ell]$, for some $a, b, c \in \C$, then $z \in \PP^1_\ell$ if and only if $(\sigma'.\sigma')=0$, i.e.\
\begin{equation}\label{equation.quadr}
2ab (\sigma.\overline{\sigma}) + c^2d=0.
\end{equation}
The only points with $c=0$ are the north and the south poles $x=[\sigma], \overline{x}=[\overline{\sigma}]$. For all the other points, $c \neq 0$ and, after rescaling $\sigma'$, we may assume $c=1$. The following result gives the explicit form of the isomorphism $\PP^1_\ell \simeq S^2_\ell$.

\begin{lemma}\label{lemma.parametrization}
	Pick an element $z=[\sigma'=a\sigma + b\overline{\sigma}+\ell] \in \PP^1_\ell$ corresponding to a point on $S^2_\ell$ different from both poles. Then its image in $S_\ell^2$ is given, in coordinates for the basis $\{\Re (\sigma), \Im (\sigma), \ell\}$, by the point
	\[
	x(a,b)=\frac{v(a,b)}{\lVert v(a,b) \rVert},
	\]
	where
	\[
	v(a,b)= \left( \Re (b-\overline{a}), \Im (b-\overline{a}), \big(a\overline{a}-b\overline{b}\big)\frac{(\sigma.\overline{\sigma})}{2d}\right)
	\]
	and $\lVert \ \ \rVert$ corresponds to $\sqrt{(\ .\ )}$.
\end{lemma}
\begin{remark}
	The basis $\{\Re (\sigma), \Im (\sigma), \ell\}$ is orthogonal for $( \ . \ )$; nonetheless, it is not normalized.
\end{remark}
\begin{remark}\label{remark.anobconj}
	If we assume $a=\overline{b}$, then, if $\sigma'\neq 0$,
	\[
	(\sigma'.\sigma')=2a\overline{a}(\sigma.\overline{\sigma})+d>0,
	\]
	as $(\sigma. \overline{\sigma})>0$. Thus $z=[\sigma'] \in \PP^1_\ell$ forces $a \neq \overline{b}$.
\end{remark}
\begin{proof}[Proof of the Lemma]
	Consider the plane $P_z=\langle \Re (\sigma'), \Im (\sigma') \rangle_\R$. Let $\alpha_z$ be the unique element of $P_T \oplus \R \ell$ of norm $1$ that is positively orthogonal to $P_z$. We are proving that $v=v(a,b)$ gives the coordinates of a vector in $P_T \oplus \R \ell$ that is positively aligned to $\alpha_z$, and this will be enough. First of all, we check the orthogonality relations. Note that
	\[
	\Re(\sigma')=\frac{\sigma'+\overline{\sigma'}}{2}=\frac{a+\overline{b}}{2}\sigma + \frac{b+\overline{a}}{2}\overline{\sigma} + \ell =\Re(b+\overline{a})\Re(\sigma)+\Im(b+\overline{a})\Im(\sigma)+\ell
	\]
	and
	\[
	\Im(\sigma')=\frac{\sigma'-\overline{\sigma'}}{2i}=\frac{a-\overline{b}}{2i}\sigma + \frac{b-\overline{a}}{2i}\overline{\sigma}=\Im\big(a-\overline{b}\big)\Re(\sigma)+\Re\big(a-\overline{b}\big)\Im(\sigma).
	\]
	Recall that
	\[
	(\Re(\sigma).\Re(\sigma))=(\Im(\sigma).\Im(\sigma))=\frac{(\sigma.\overline{\sigma})}{2}, \quad (\Re(\sigma).\Im(\sigma))=0.
	\]
	Then, simple computations give
	\[
	(v.\Re(\sigma'))=\frac{(\sigma.\overline{\sigma})}{2}\big(\Re(b-\overline{a})\Re(b+\overline{a})+\Im(b-\overline{a})\Im(b+\overline{a})+\big(a\overline{a}-b\overline{b}\big)\big)=0
	\]
	and
	\[
	(v.\Im(\sigma'))=\frac{(\sigma.\overline{\sigma})}{2}\big(\Re(b-\overline{a})\Im\big(a-\overline{b}\big) + \Im(b-\overline{a})\Re\big(a-\overline{b}\big)\big)=0.
	\]
	The positive alignment follows from the positivity of the discriminant of the matrix
	\[
	\begin{bmatrix}
	\Re(b+\overline{a}) & \Im(b+\overline{a}) & 1 \\
	\Im\big(a-\overline{b}\big) & \Re\big(a-\overline{b}\big) & 0 \\
	\Re (b-\overline{a}) & \Im (b-\overline{a}) & \big(a\overline{a}-b\overline{b}\big)\frac{(\sigma.\overline{\sigma})}{2d}
	\end{bmatrix}
	\]
	which is
	\[
	\big(a\overline{a}-b\overline{b}\big)^2\frac{(\sigma.\overline{\sigma})}{2d}+\big(a-\overline{b}\big)\overline{\big(a-\overline{b}\big)}>0
	\]
	as $a\neq \overline{b}$, thanks to Remark \ref{remark.anobconj}.
\end{proof}

\begin{remark}\label{remark.N(a)=N(b)}
	The equator $S^1_\ell$ is defined by the condition $|a|=|b|$, or equivalently $a\overline{a}=b\overline{b}$, for $z=[\sigma'=a \sigma + b \overline{\sigma} + \ell]$.
\end{remark}

\begin{lemma}\label{lemma.twopoints}
	Choose a non-zero element $\ell' \in T \oplus \Q \ell$. Then there are exactly two points $z, z' \in \PP^1_\ell$ such that $\ell'$ is orthogonal to $z$ and $z'$, i.e.\ $\ell' \in P_z^\perp$ and $\ell' \in P_{z'}^\perp$. Moreover, $z'=\overline{z}$, and $z$ and $\overline{z}$ correspond to antipodal points on $S^2_\ell$ via the isomorphism $\PP^1_\ell \simeq S^2_\ell$.
\end{lemma}

\begin{proof}
	The case $\ell'=\ell$ has already been discussed; suppose that $\ell' \notin \Q \ell$. The orthogonal complement of $\ell'$ in $T_\C \oplus \C \ell$ does not contain the whole $T^{2,0} \oplus T^{0,2} \oplus \C \ell$ (since the form $( \ . \ )$ is non-degenerate on this last vector space). Therefore $\ell'^\perp \cap (T^{2,0} \oplus T^{0,2} \oplus \C \ell)$ is a $\C$-plane in $T^{2,0} \oplus T^{0,2} \oplus \C \ell$. Passing to the projective spaces, $\ell'^\perp \cap (T^{2,0} \oplus T^{0,2} \oplus \C \ell)$ defines a line in $\PP(T^{2,0} \oplus T^{0,2} \oplus \C \ell)$, which cuts the conic $\PP^1_\ell$ in two distinct points, or one single point (with double multiplicity). On the other hand, the second case is not admissible: if $z=[\sigma'=a\sigma + b\overline{\sigma} +\ell]$ satisfies conditions (\ref{equation.quadr}) and (\ref{equation.lin}), then
	\[
	\overline{z}=\left[\overline{\sigma'}=\overline{b}\sigma + \overline{a} \overline{\sigma} +\ell\right]
	\]
	satisfies both equations as well. Besides, as already pointed out in Remark \ref{remark.anobconj}, $a \neq \overline{b}$, so that $z \neq \overline{z}$. Then $z, \overline{z}$ are the required points. Lastly, note that $z$ and $\overline{z}$ are antipodal on the sphere $\PP^1_\ell \simeq S^2_\ell$; for, compare with the isomorphism of Lemma \ref{lemma.parametrization}.
\end{proof}

\section{Picard jump on the equator}\label{section.excessive}

We focus our attention on the Noether-Lefschetz locus of the equator $S^1_\ell \subseteq \PP^1_\ell$, its points of Picard jump and their period fields.

\subsection{Excessive Picard jump values}\label{section.values}

We have already remarked that the original extended Hodge structure of K3 type on $T \oplus \Q \ell$ is no longer irreducible, being $\ell$ a $(1,1)$-class. Given $z \in \PP^1_\ell$, we define the \emph{Picard number} $\rho_z$ of the Hodge structure corresponding to $z$ to be the $\Q$-dimension of the space of $(1,1)$-classes of $T \oplus \Q \ell$, that is
\[
\rho_z = \dim_\Q \big(P_z^\perp \cap (T \oplus \Q \ell)\big).
\]
For instance, as the original $T$ is irreducible, we deduce that $\rho_x = \rho_{\overline{x}} = 1$.

\begin{definition}
	We say that a point $z \in \PP^1_\ell$ is of \emph{Picard jump} if $\rho_z \ge 1$, of \emph{excessive Picard jump} if $\rho_z > 1$. We call \emph{Noehter-Lefschetz locus} the set of points $z \in \PP^1_\ell$ of Picard jump.
\end{definition}

We are interested in understanding how points of Picard jump distribute on the sphere $\PP^1_\ell \simeq S^2_\ell$, and the possible relative Picard numbers. Huybrechts proved the following result \cite[Proposition 3.2]{Huy}.

\begin{proposition}\label{proposition.jump}
	Assume that $T$ is a polarized irreducible Hodge structure of K3 type. Then, for the twistor base $\PP^1_\ell \simeq S^2_\ell$, one has:
	\begin{enumerate}
		\item the set $\Set{z \in \PP^1_\ell | \rho_z \ge 1}$ is countable and dense (in the classical topology);
		\item the set $\Set{z \in \PP^1_\ell | \rho_z > 1}$ is at most countable and contained in the equator $S^1_\ell$.
	\end{enumerate}
\end{proposition}

The fact that the Noether-Lefschetz locus is countable is a particular instance of a more general fact (see \cite[Chapter 6, Proposition 2.9]{HuyK3}). However, there is a very simple argument that applies in this case. Assume that $z$ is of Picard jump, and let $\ell' \in T \oplus \Q \ell$ be a $(1,1)$-class for the Hodge structure induced by $z$. Then $\ell'^\perp$ identifies a line in $\PP(T^{2,0}\oplus T^{0,2} \oplus \C \ell)$, that cuts the conic $\PP^1_\ell$ in two points: $z$ and $\overline{z}$. Thus, $\ell'$ is a $(1,1)$-class only for finitely many $z \in \PP^1_\ell$, and therefore the Noether-Lefschetz locus is countable, as $T \oplus \Q \ell$ is.

\begin{center}
	\includegraphics[scale=0.5]{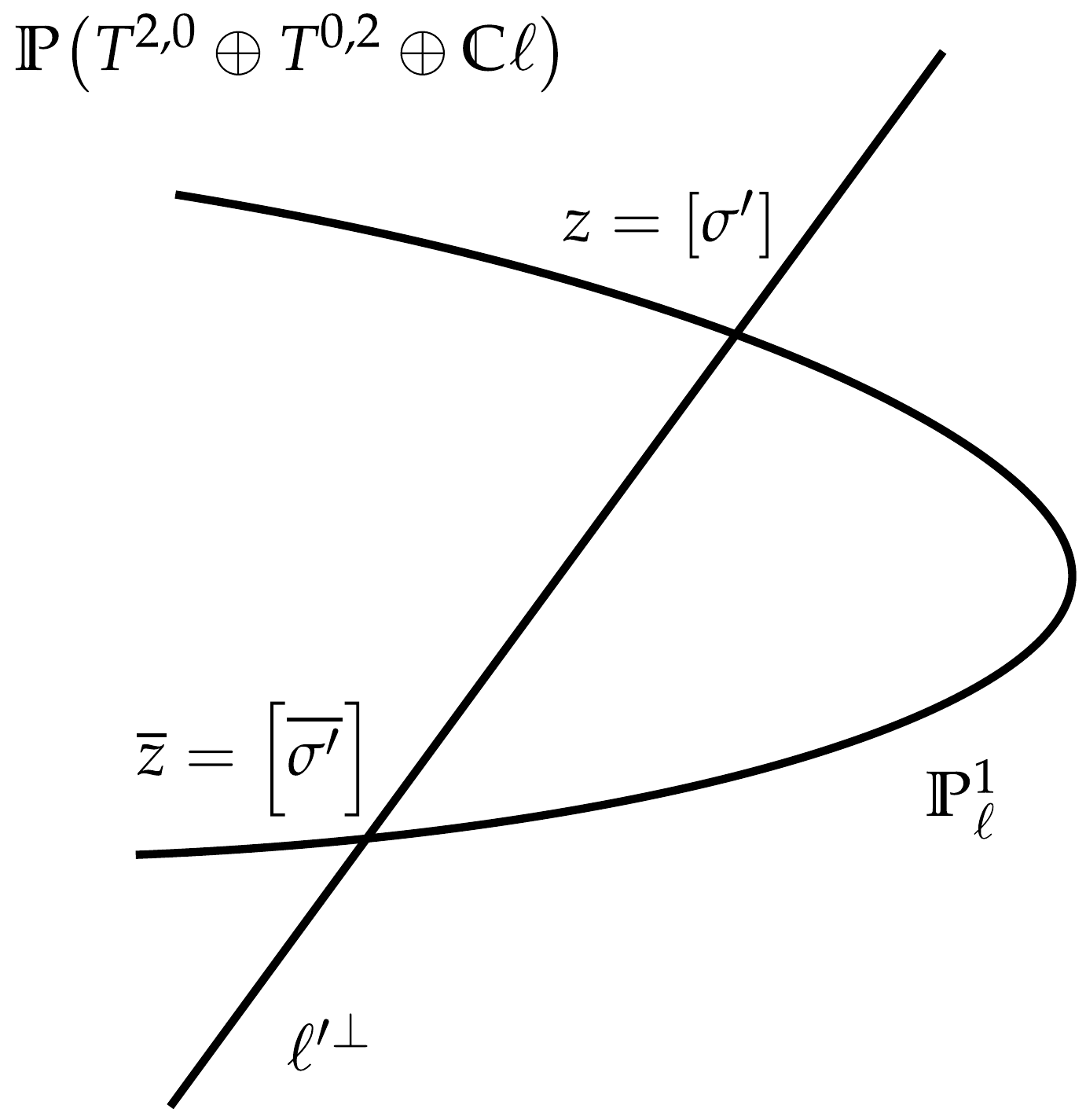}
\end{center}

Assume that $z=[a\sigma + b\overline{\sigma} + c \ell]$. The condition defining $\ell'^\perp$ can be explicitly translated as
\begin{equation}\label{equation.lin}
a(\sigma. \ell')+b(\overline{\sigma}.\ell')+(\ell.\ell')=0.
\end{equation}

\begin{lemma}\label{lemma.equatorT}
	Assume that $\ell' \in T \oplus \Q \ell$, $\ell' \notin \Q \ell$ and choose $z$ to be one of the two points orthogonal to $\ell'$. Then the point $z$ is contained in the equator $S^1_\ell$ if and only if $\ell' \in T$.
\end{lemma}
\begin{proof}
	Write $z=[\sigma']$. Assume that $z \in S^1_\ell$, i.e.\ $\ell \in P_{z}=\langle \Re(\sigma'), \Im(\sigma') \rangle_\R \subseteq \big\langle \sigma', \overline{\sigma'}\big\rangle_\C$. Then $\ell=a\sigma' +b \overline{\sigma'}$ for some $a, b \in \C$; hence $(\ell.\ell')=0$ since $(\sigma'.\ell')=0$, which means $\ell' \in T$. Conversely, assume that $\ell' \in T$, i.e.\ $(\ell.\ell')=0$. Write $\sigma'=a\sigma+b\overline{\sigma}+c\ell$. After rescaling $\sigma$, we may assume that $(\sigma.\ell')=1$, so that $(\overline{\sigma}.\ell')=\overline{(\sigma.\ell')}=1$ as well. The condition $(\sigma'.\ell')=0$ implies,
	\[
	a(\sigma.\ell')+b(\overline{\sigma}.\ell')+c(\ell.\ell')=0, \ \text{i.e.} \ b=-a.
	\]
	Therefore, the condition (\ref{equation.quadr})
	\[
	-2a^2 (\sigma.\overline{\sigma})+c^2d=0
	\] implies $c\neq 0$ (and therefore we may assume $c=1$) and $a \in \R$ (since $d>0$ and $(\sigma.\overline{\sigma})>0$). Therefore
	\[
	\overline{\sigma'}=a\overline{\sigma}-a\sigma +\ell
	\]
	and
	\[
	\ell=\frac{\sigma'+\overline{\sigma'}}{2} = \Re (\sigma')
	\]
	and the proof is concluded, as $\ell \in P_{z}$.
\end{proof}

If the original Hodge structure $T$ has CM, all points of Picard jump on the equator share the same Picard number.

\begin{proposition}\label{proposition.picardjumpeq}
	Suppose that $T$ is a polarized irreducible Hodge structure of K3 type and CM, and consider its sphere of related Hodge structures. Assume that $z \in S^1_\ell$ is a point of Picard jump on the equator. Then $\rho_z=r/2$.
\end{proposition}

\begin{center}
	\vcenteredhbox{\begin{tabular}{cc}
			\toprule
			&
			\begin{tabular}{c}
				Admissible values \\
				of $\rho_z$, CM case \\
			\end{tabular}
			\\
			\toprule
			\begin{tabular}{c}
				Outside the \\
				equator \\
			\end{tabular}
			& $0, 1$ \\
			\midrule
			\begin{tabular}{c}
				On the \\
				equator \\
			\end{tabular}
			& $0, r/2$ \\
			\bottomrule
	\end{tabular}}
	\qquad \qquad
	\vcenteredhbox{\includegraphics[width=0.3\textwidth]{sphere}}
\end{center}

\begin{proof}
	Firstly, note that $P_{z}^\perp \cap (T\oplus\Q \ell)$ is all contained in $T$. This follows from Lemma \ref{lemma.equatorT} applied to any non-zero element $\ell' \in P_{z}^\perp \cap (T\oplus\Q \ell)$. Hence $P_{z}^\perp \cap (T\oplus\Q \ell)=P_{z}^\perp \cap T$. The key remark is that $P_{z}^\perp \cap T$ admits an action of the real part $K_T^0=K_T \cap \R$ of the endomorphism field $K_T$. To show this, choose $\beta \in K_T^0$ and $\ell' \in P_{z}^\perp \cap T$. Write $z=[\sigma'=a\sigma + b\overline{\sigma}+\ell]$. Then, using that $(\ell.\ell')=0$, we obtain ($\beta \in K_T^0$ is self-adjoint for $( \ . \ )$)
	\begin{align*}
	(\sigma'.\beta(\ell'))&=(\beta(a\sigma+b\overline{\sigma}+\ell).\ell')\\
	&=a (\beta \cdot \sigma . \ell') + b (\beta \cdot \overline{\sigma} . \ell') + 0 \\
	&=\beta \cdot \big( a (\sigma . \ell') + b (\overline{\sigma} . \ell') + 0 \big) \\
	&= \beta \cdot ( \sigma'.\ell') =0.
	\end{align*}
	As a consequence, $P_{z}^\perp \cap T$ is a $K_T^0$-vector space. As $[K_T^0 \colon \Q]=r/2$ (we are in the CM case by assumption, so that $[K_T\colon K_T^0]=2$ and $[K_T : \Q]= r$), $\dim_{K_T^0} P_{z}^\perp \cap T$ can only take three different values: $0, 1, 2$. We immediately exclude the case $\dim_{K_T^0} P_{z}^\perp \cap T=0$ since $z$ is a point of Picard jump. Assume by contradiction that $\dim_{K_T^0} P_{z}^\perp \cap T=2$. In fact, this equality would imply $P_{z}^\perp \cap T=T$. This would force $\sigma'$ to belong to $\C \ell$, or equivalently to have zero $T_\C$-part (indeed the form $( \ . \ )$ in non-degenerate on $T$, and therefore on $T_\C$). On the other hand, this is a contradiction, since any non-zero $\sigma'\in \C \ell$ does not satisfy (\ref{equation.quadr}). In conclusion, $\dim_{K_T^0} P_{z}^\perp \cap T=1$, or equivalently $\dim_\Q P_{z}^\perp \cap T= [K_T^0 \colon \Q]=r/2$, i.e.\ $\rho_z=r/2$.
\end{proof}

\begin{remark}\label{remark.noirreducible}
	If the original $T$ has CM, the only case where there are no points of excessive Picard jump (i.e.\ $\rho_z \le 1$ for all $z \in \PP^1_\ell$) is $r=2$.
\end{remark}

\begin{remark}
	The hypothesis on $T$ of being of CM type plays a fundamental role in order to ensure the validity of the property $\rho_z = r/2$ for all points of jump on the equator. It is not difficult to construct explicit examples of irreducible polarized Hodge structures of K3 type $T$ for which the excessive Picard values are multiple. In addition, not even the sole assumption that $K_T$ is CM, without supposing $\dim_\Q T =[K_T : \Q]$, is enough to guarantee the result.
\end{remark}

\begin{remark}
	However, if $T$ is not assumed to be of CM type, one may argue as in the proof of Proposition \ref{proposition.picardjumpeq} to deduce that $\rho_z$ (for $z$ point of Picard jump on the equator) is divisible by $[K_T^0 : \Q]$ and strictly smaller than $r$.
\end{remark}

\begin{center}
	\vcenteredhbox{\begin{tabular}{cc}
			\toprule
			&
			\begin{tabular}{c}
				Admissible values \\
				of $\rho_z$, non-CM case \\
			\end{tabular}
			\\
			\toprule
			\begin{tabular}{c}
				Outside the \\
				equator \\
			\end{tabular}
			& $0, 1$ \\
			\midrule
			\begin{tabular}{c}
				On the \\
				equator \\
			\end{tabular}
			&
			\begin{tabular}{c}
				$0$, $d$ such that \\
				$[K_T^0 : \Q] \mid d$ and $d < r$ \\
			\end{tabular}
			\\
			\bottomrule
	\end{tabular}}
	\qquad \qquad
	\vcenteredhbox{\includegraphics[width=0.3\textwidth]{sphere}}
\end{center}

The following result enriches Proposition \ref{proposition.jump}. To prove it, it is not necessary to assume that $T$ has CM, or that $K_T$ is CM. However, in Section \ref{section.actionequator} we will give a better description of the distribution of points of Picard jump on the equator in the CM case.

\begin{proposition}\label{proposition.jumpdenseoneq}
	Assume that $T$ is a polarized irreducible Hodge structure of K3 type, and consider its sphere of related Hodge structures. Then the set of points of Picard jump on the equator is dense in the equator (for the classical topology). In particular, this set is countable (not finite).
\end{proposition}

\begin{proof}
	Thanks to Lemma \ref{lemma.parametrization}, a point on the equator $S^1_\ell \subseteq \PP^1_\ell$ corresponds via the isomorphism $\PP^1_\ell \simeq S^2_\ell$ to the normalization of the vector
	\[
	v(a,b)= \big( \Re (b-\overline{a}), \Im (b-\overline{a}), 0 \big).
	\]
	Lemma \ref{lemma.parametrization} also gives a condition on $z$ to belong to $S^1_\ell$, namely $\abs{a}=\abs{b}$. Write $R$ to denote these absolute values; then $a=Re^{i\theta}, b=Re^{i\tau}$ for some $\theta, \tau \in \R$. Moreover, (\ref{equation.quadr}) gives
	\[
	2 R^2 e^{i (\theta + \tau)} (\sigma.\overline{\sigma}) = -d,
	\]
	from which we deduce that $\theta+\tau\equiv_{2\pi}\pi$. Hence, $b=-Re^{-i\theta}=-\overline{a}$. Then
	\[
	v(a,b)= \big( -2 R \cos(-\theta), -2 R \sin(-\theta) , 0 \big).
	\]
	Therefore, in order to prove the statement, it is enough to show that the set of possible complex arguments assumed by the periods $(\sigma.\gamma)$, $\gamma \in T$ is dense in the circle $\{ z \in \C \ | \ \abs{z}=1 \}$ (the point of Picard jump corresponding to such a $\gamma$ would be determined by $a = R e^{i\theta}$ such that $\theta$ (or $\pi+\theta$) is the opposite of the argument of $(\sigma.\gamma)$). However, the $\Q$-vector space $P$ of the periods $(\sigma.\gamma)$, $\gamma \in T$ is not contained in an $\R$-line of $\C$; indeed, if this were the case, then the period field $k_T = \Q((\sigma.\gamma_i)/(\sigma.\gamma_1))$ (for a basis $\{\gamma_i\}$ of $T$) would be contained in $\R$, contradiction. As a consequence, $P$ is dense in $\C$, and then the set of the complex arguments of elements of $P\setminus \{0\}$ is dense in the circle.
\end{proof}

\section{Actions on Noether-Lefschetz loci}\label{section.actions}

In this section, assume that $T$ is a polarized irreducible Hodge structure of K3 type, and consider again the associated twistor base $\PP^1_\ell \simeq S^2_\ell$. We define two actions: first, an action of the multiplicative group $K_T^\times$ on the Noether-Lefschetz locus of the upper-half sphere, deprived of the north pole; then, an action of the multiplicative group $K_T^\times/(K_T^0)^\times$ on the Noether-Lefschetz locus of the equator. We treat in particular the situation arising when we assume that $T$ is of CM type: the jumping loci outside and on the equator are homogeneous under the actions of $K_T^\times$ and $K_T^\times/(K_T^0)^\times$, respectively.

\medskip

One may define these two actions together as a unique action on the whole $\PP^1_\ell$; however, we prefer to keep them separated to analyse the different behaviours on points outside and on the equator $S^1_\ell$.

\subsection{Outside the equator}\label{section.outsideequator}

Denote by $\U$ the upper-half sphere of $\PP^1_\ell \simeq S^2_\ell$ (defined by the condition $\abs{a}>\abs{b}$ on $z=[\sigma'=a\sigma+b\overline{\sigma}+\ell]$, see Lemma \ref{lemma.parametrization}), and by $\QQ$ the set of points of Picard jump in $\U$ (that is, the Noether-Lefschetz locus of $\U$), deprived of the north pole $x$. Notice that a point $z \in \QQ$ satisfies $\rho_z=1$, by Proposition \ref{proposition.jump}, and hence is orthogonal to a unique (up to a rational scalar) non-zero element $\ell' \in P_z^\perp \cap (T \oplus \Q \ell)$. Denote by $\QQ^+$ the set of $z \in \QQ$ for which the corresponding $\ell'$ is such that $(\ell'.\ell')>0$.

\begin{remark}\label{remark.QneqQ+}
	Unless $r=2$, the form $( \ . \ )$ is not positive definite on $T$. Thus, $\QQ^+ \subsetneq \QQ$ if $r>2$.
\end{remark}

For a point $z \in \PP^1_\ell$, we define its \emph{altitude} as the last coordinate of the vector defined by Lemma \ref{lemma.parametrization}, namely the third coordinate of the corresponding point in $S^2_\ell$ for the basis $\{\Re (\sigma), \Im (\sigma), \ell\}$ of $P_T \oplus \R \ell$.

\medskip

Recall that $d=(\ell.\ell) \in \Z_{>0}$. Choose an element $\ell' = \gamma_1+(m/d) \ell \in T \oplus \Q \ell$, $\ell' \notin T$, $\ell' \notin \Q \ell$, so that the two corresponding orthogonal points are neither the poles nor on the equator (see Lemma \ref{lemma.twopoints} and Lemma \ref{lemma.equatorT}). Denote by $z_1 \in \QQ$ the only point of Picard jump in the upper-half sphere orthogonal to $\ell'$. Taking a scalar multiple of $\ell'$ identifies the same $z_1$; we may assume $m=d$, so that $\ell'=\gamma_1+\ell$. For an element $A \in K_T^\times$, we define $z_2=A * z_1$, where $z_2 \in \QQ$ is the only point of Picard jump in the upper-half sphere orthogonal to $\ell''=\gamma_2+\ell$, where $\gamma_2=A(\gamma_1)$. This association defines an action of $K_T^\times$ on $\QQ$.

\begin{proposition}
	The action defined above is free. Moreover, if $T$ has CM, the action is transitive.
\end{proposition}

\begin{proof}
	If $A \in K_T^\times$ is different from the identity morphism, $\ell'$ and $\ell''$ are linearly independent over $\Q$, so that they cannot be orthogonal to the same point in $\QQ$ (thanks to Proposition \ref{proposition.jump}); thus, the action is free. The transitivity of the action in the CM case follows from $\dim_{K_T} T =1$.
\end{proof}

\begin{remark}
	If $T$ does not have CM, the action is no longer transitive (as $\dim_{K_T} T > 1$).
\end{remark}

\begin{proposition}\label{proposition.altitude}
	Under the same assumptions as above, suppose that $z_1=[\sigma'=a\sigma + b\overline{\sigma} +\ell] \in \QQ$ and additionally that $A \in K_T^\times$. Then
	\begin{enumerate}
		\item If $\, \abs{A}=1$, say $A=e^{i\theta}$, then $z_2= A * z_1$ is represented by the element
		\[
		\sigma''=aA(\sigma)+bA(\overline{\sigma})+\ell=aA\sigma+b\overline{A}\overline{\sigma}+\ell=ae^{i\theta}\sigma+be^{-i\theta}\overline{\sigma}+\ell.
		\]
		In particular, $A$ acts on $\QQ$ by a rotation of angle $-\theta$ along the $\ell$-axis, and $z_1$ and $z_2$ have the same altitude;
		\begin{center}
			\includegraphics[width=0.3\textwidth]{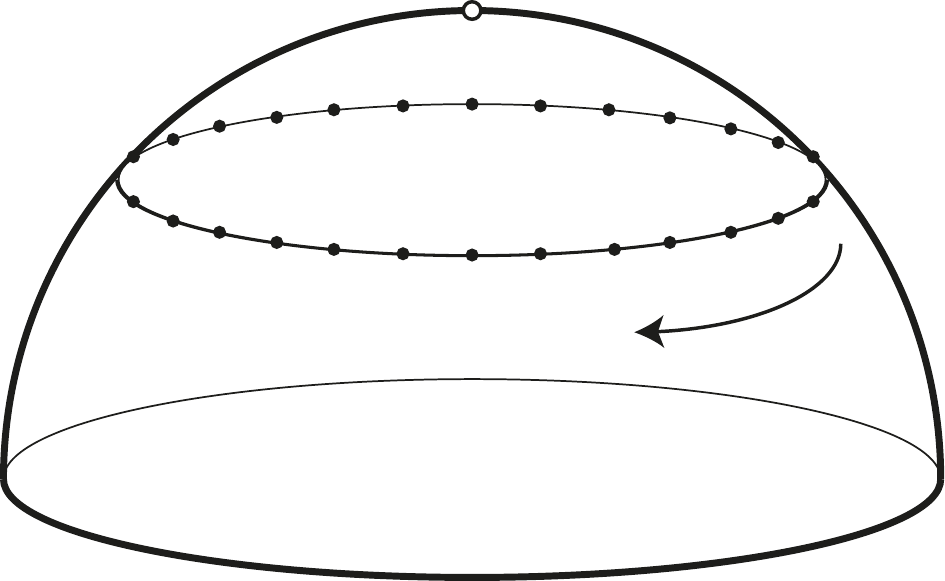}
		\end{center}
		\item If $\, \abs{A} \neq 1$, then $z_1$ and $z_2=A*z_1$ do not have the same altitude. In particular, if $\, \abs{A} < 1$ the altitude of $z_2$ is greater than the one of $z_1$, and vice versa;
		\begin{center}
			\includegraphics[width=0.3\textwidth]{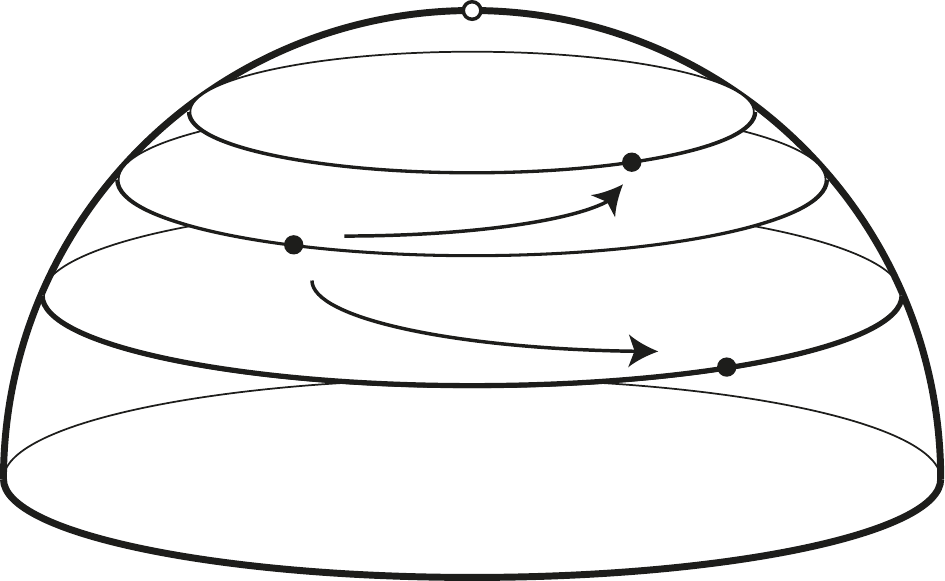}
		\end{center}
		\item If $A \in K_T^0 = K_T \cap \R$, then $z_1$ and $z_2=A*z_1$ lie on the same meridian of $\PP^1_\ell \simeq S^2_\ell$.
		\begin{center}
			\includegraphics[width=0.3\textwidth]{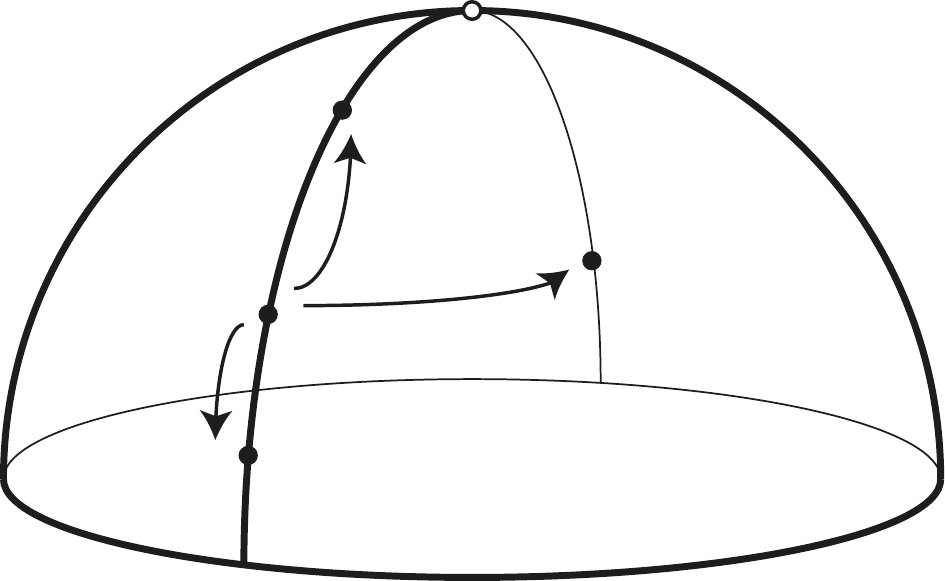}
		\end{center}
	\end{enumerate}
\end{proposition}

\begin{remark}\label{remark.N}
	Recall that the homeomorphism $\PP^1_\ell \simeq S^2_\ell$ of Lemma \ref{lemma.parametrization} sends $z_1=[\sigma'=a \sigma+ b \overline{\sigma} + \ell]$ to the normalization of the vector
	\[
	v(a,b)= \left( \Re (b-\overline{a}), \Im (b-\overline{a}), \big(a\overline{a}-b\overline{b}\big)\frac{(\sigma.\overline{\sigma})}{2d}\right).
	\]
	This allows us to deduce that the tangent of the angle between a vector of a point $S^2_\ell$ and the plane of zero-altitude is given, up to sign, by
	\[
	\frac{(\sigma.\overline{\sigma})}{2d}\abs{\overline{a}+b}.
	\]
	Indeed, it is suffices to notice that
	\[
	a\overline{a}-b\overline{b}=(\overline{a}+b)(a-\overline{b}),
	\]
	since $ab \in \R$ by (\ref{equation.quadr}).
\end{remark}

\begin{proof}[Proof of the Proposition]
	Firstly, assume that $\abs{A}=1$. We are showing that $z_2=[\sigma'']$ actually belongs to the conic $\PP_\ell^1$, i.e.\ $(\sigma''.\sigma'')=0$. For, it is enough to observe that the (\ref{equation.quadr}), that here takes the form
	\[
	2(aA)\big(b\overline{A}\big)(\sigma.\overline{\sigma})+d=0,
	\]
	is satisfied since $\abs{A}=1$ and $z_1 \in \PP^1_\ell$. If $\abs{a}>\abs{b}$ then $\abs{aA}>\abs{b\overline{A}}$, so that $z_2$ still belongs to the upper-half sphere. To prove the first statement, it suffices to show that $(\sigma''.\ell'')=0$ holds true, where $\ell''=A(\gamma_1)+\ell$. Recall that, since $\abs{A}=1$, $A$ is an isometry for $( \ . \ )$; therefore we obtain
	\begin{align*}
	(\sigma''.\ell'')&=\big(aA(\sigma)+bA(\overline{\sigma})+\ell.A(\gamma_1)+\ell\big)\\
	&=a(A(\sigma).A(\gamma_1))+b(A(\overline{\sigma}).A(\gamma_1))+d\\ &=a(\sigma.\gamma_1)+b(\overline{\sigma}.\gamma_1)+d\\
	&=(\sigma'.\ell')=0.
	\end{align*}
	For the statement concerning the action by rotation and the same altitude of $z_1$ and $z_2$, it is enough to compare the tangents of the angles between the vectors corresponding to $z_1$ and $z_2$, as done in Remark \ref{remark.N}.
	
	\begin{remark}\label{remark.nosigma''}
		One could expect, for a general $A \in K_T^\times$ not necessarily of norm one, the point $z_2= A * z_1$ to be given by
		\[
		z_2=\left[\sigma''=a\overline{A}^{-1}(\sigma)+b\overline{A}^{-1}(\overline{\sigma})+\ell=a\overline{A}^{-1}\sigma+bA^{-1}\overline{\sigma}+\ell\right],
		\]
		so that $(\sigma''.\ell'')=0$ by a similar argument. Nonetheless, (\ref{equation.quadr}) is no longer satisfied and $z_2 \notin \PP^1_l$, so this guess is not true. It is indeed more complicated to deduce an explicit expression for $z_2$ when $A$ does not have norm one. 
	\end{remark}
	
	We continue the proof of Proposition \ref{proposition.altitude}. Assume now that $\abs{A} \neq 1$. Even without computing explicitly $z_2$ we can prove that $z_2$ has different altitude from $z_1$. Define $B=\overline{A}^{-1}$ and consider
	\[
	\sigma''=aB\sigma + b\overline{B}\overline{\sigma}+\ell.
	\]
	It corresponds to a point in $\PP(T^{2,0}\oplus T^{0,2}\oplus \C \ell)$ orthogonal to $\ell''$. Then $0=\overline{(\sigma''.\ell'')}=\big(\overline{\sigma''}.\ell''\big)$, where
	\[
	\overline{\sigma''}=\overline{b}B\sigma+\overline{a}\overline{B}\overline{\sigma}+\ell.
	\]
	Note that $\sigma'', \overline{\sigma''}$ are linearly independent over $\C$ since $a \neq \overline{b}$ (see Remark \ref{remark.anobconj}). Therefore, the line in $\PP(T^{2,0}\oplus T^{0,2}\oplus \C \ell)$ passing through these points is exactly the line of elements orthogonal to $\ell''$; among these points there are $z_2$ and $\overline{z_2}$, given by the intersection with the conic $\PP^1_\ell$. The general point of this line has the form $z=\big[\lambda \sigma'' + \mu \overline{\sigma''}\big]$. Since $\lambda+\mu=0$ does not give an element of $\PP^1_\ell$, we may assume $\lambda+\mu=1$, and $z=\big[\lambda \sigma'' + (1-\lambda) \overline{\sigma''}\big]$.
	
	\begin{center}
		\includegraphics[width=0.6\textwidth]{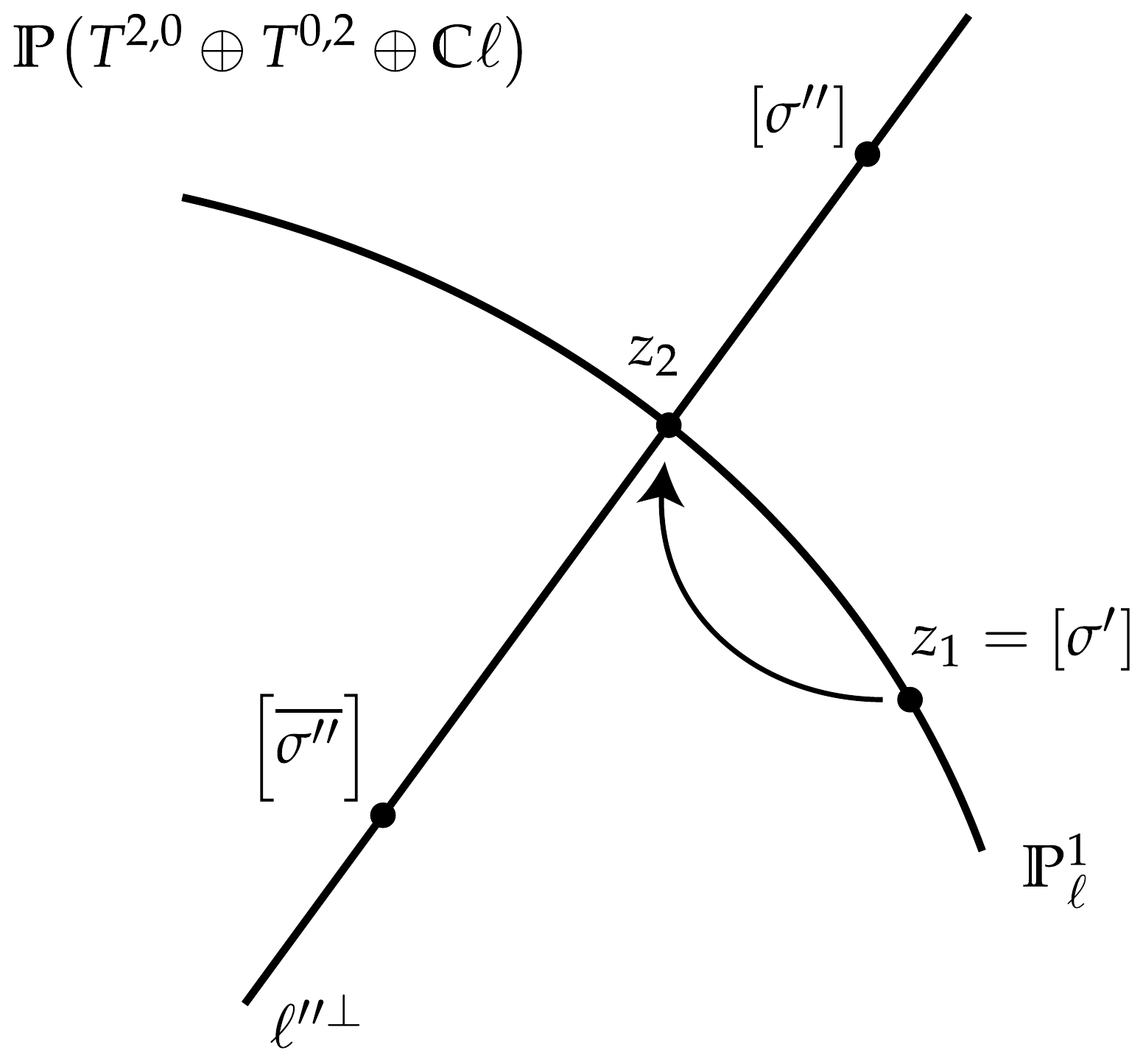}
	\end{center}
	
	Assume that $z \in \PP^1_\ell$. We claim that $\lambda \in \R$. Firstly, note that $\lambda \neq 1/2$, otherwise the $\sigma$-coefficient and the $\overline{\sigma}$-coefficient of $\lambda \sigma'' + (1-\lambda) \overline{\sigma''}$ would have the same norm, forcing $z$ to be on the equator $S^1_\ell$, contradiction. Since $z \in \PP^1_\ell$, then
	\[
	0=\big(\lambda \sigma'' + (1-\lambda) \overline{\sigma''}\big)^2=\lambda^2 (\sigma''.\sigma'')+2\lambda\big(1-\lambda)(\sigma''.\overline{\sigma''}\big)+(1-\lambda)^2\big(\overline{\sigma''}.\overline{\sigma''}\big).
	\]
	Note that $(\sigma''.\sigma'')=\big(\overline{\sigma''}.\overline{\sigma''}\big)$. Therefore, also the point $z'$ identified by
	\[
	z' = \left[ (1-\lambda)\sigma''+\lambda \overline{\sigma''} \right]
	\]
	is in the intersection of the conic $\PP^1_\ell$ with the line of elements orthogonal to $\ell''$. Since $\lambda \neq 1/2$, it cannot coincide with $z$, and thus $z'=\overline{z}$, where
	\[
	\overline{z}=\left[\overline{\lambda  \sigma''}+\big(1-\overline{\lambda}\big)\sigma''\right].
	\]
	Hence, $\lambda=\overline{\lambda}$.
	
	\medskip
	
	As noticed in Remark \ref{remark.N}, to discuss the difference of altitude we are interested in the quantity $\abs{\overline{a}+b}$ for the new coefficients. We have
	\[
	\lambda\sigma''+\big(1-\lambda\overline{\sigma''}\big)=\big(\lambda a +(1-\lambda)\overline{b}\big)B \sigma + \big(\lambda b + (1-\lambda) \overline{a}\big)\overline{B} \overline{\sigma} +\ell
	\]
	and, using $\lambda \in \R$,
	\[
	\abs{ \big(\lambda \overline{a} + (1-\lambda) b \big) \overline{B} + \big(\lambda b + (1-\lambda) \overline{a}\big)\overline{B} } = \abs{\overline{a}+b} \abs{\overline{B}} = \abs{\overline{a}+b} \abs{A}^{-1}.
	\]
	This relation tells us exactly that, for points in the upper-half sphere, acting by an element of $K_T^\times$ of norm smaller than one increases the altitude, while the action of elements of $K_T^\times$ of norm greater than one decreases the altitude.
	
	\medskip
	
	Finally, we prove the third statement. If $A \in \R$ then $B=\overline{A}^{-1}=A^{-1} \in \R$, too. Arguing as above, we deduce that $z_2$ is represented by the element
	\[
	\lambda \sigma'' + (1-\lambda) \overline{\sigma''} = \big(\lambda a + (1-\lambda)\overline{b}\big) B \sigma + \big(\lambda b + (1-\lambda) \overline{a}\big) B \overline{\sigma} + \ell
	\]
	for some $\lambda \in \R$. Set
	\[
	\tilde{a}=\big(\lambda a + (1-\lambda)\overline{b}\big) B, \quad \tilde{b}=\big(\lambda b + (1-\lambda) \overline{a}\big) B.
	\]
	For this element, the quantity $\tilde{b}-\overline{\tilde{a}}$ is equal to
	\[
	\left( (2\lambda-1) b + (1-2\lambda) \overline{a} \right) B = (b-\overline{a}) (2\lambda-1) B.
	\]
	As $(2\lambda-1) B \in \R$, the angle determined by $\Re\big(\tilde{b}-\overline{\tilde{a}}\big)$ and $\Im\big(\tilde{b}-\overline{\tilde{a}}\big)$ is the same as the angle determined by $\Re(b-\overline{a})$ and $\Im(b-\overline{a})$, so that $z_2$ and $z_1$ lie on the same meridian.
\end{proof}

\begin{remark}
	Of course, one can state a similar result for points in the lower-half sphere, as the lower-half may be obtained by conjugating the upper-half sphere.
\end{remark}

Huybrechts proved that, for a point $z_1 \in \QQ^+$ orthogonal to $\ell'$, the corresponding polarized irreducible Hodge structure of K3 type, namely $T'=\ell'^\perp$, is of CM type, and the real parts $K_T^0$ and $K_{T'}^0$ of $K_T$ and $K_{T'}$, respectively, coincide \cite[Proposition 3.8]{Huy}. Moreover, he gave an explicit description of $K_{T'}$ \cite[Corollary 3.10]{Huy}: it is the quadratic extension of $K_T^0=K_{T'}^0$ described by
\[
X^2 + \gamma X + \delta =0,
\]
where\footnote{In \cite[Corollary 3.10]{Huy} a factor $2$ in the expression of $\delta$ is missing.}
\[
\gamma = m(\alpha +\alpha^{-1}), \quad \delta =m^2-\frac{d}{2(\sigma . \overline{\sigma})}(\alpha^2+\alpha^{-2}-2),
\]
$K_T=\Q(\alpha)$, $d=(\ell.\ell)$, $m=(\ell.\ell')$ and $\sigma$ is such that $(\sigma.\ell')=1$.

\begin{proposition}\label{proposition.sameCMfield}
	Same hypotheses as in Proposition \ref{proposition.altitude}. Suppose that $T$ has CM. Assume, moreover, that $z_1 \in \QQ^+$ and that $A \in K_T^\times$ satisfies $\abs{A}=1$. Then $z_2=A*z_1 \in \QQ^+$ and the polarized irreducible Hodge structures of K3 type corresponding to the points $z_1$ and $z_2$, namely $\ell'^\perp$ and $\ell''^\perp$, have the same CM field.
\end{proposition}

\begin{proof}
	To prove that $z_2 \in \QQ^+$, it is enough to notice that
	\begin{align*}
	(\ell''.\ell'')&= (A(\gamma_1)+\ell.A(\gamma_1)+\ell)\\
	&= (A(\gamma_1).A(\gamma_1)) + (\ell.\ell)\\
	&= (\gamma_1.\gamma_1)+ (\ell.\ell)\\
	&=(\ell'.\ell')>0,
	\end{align*}
	as $A$ is an isometry for $( \ . \ )$ and $z_1 \in \QQ^+$.
	
	\medskip
	
	We now prove that the coefficients $\gamma, \delta$ are the same for the two points $z_1, z_2$, so that $K_{T'}=K_{T''}$. For both points $m=d$, since we fixed the $\ell$-part of $\ell', \ell''$; thus $\gamma$ is the same. Note that $\sigma$ depends on $\ell'$: it is chosen such that $(\sigma.\ell')=1$. On the other hand, if we choose such a $\sigma$ for $\ell'$, then $A(\sigma)=A\sigma$ satisfies
	\[
	(A(\sigma).\ell'')=(A(\sigma).A(\gamma_1)+\ell)=(A(\sigma).A(\gamma_1))=(\sigma.\gamma_1)=(\sigma.\ell')=1,
	\]
	since $A$ is an isometry. Hence, we may choose $A(\sigma)=A\sigma$ for $\ell''$. Then
	\[
	\big(A\sigma . \overline{A \sigma}\big)=\big(A\sigma . \overline{A} \overline{\sigma}\big) = A\overline{A}(\sigma.\overline{\sigma})=(\sigma.\overline{\sigma})
	\]
	so that $\delta$ is the same as well.
\end{proof}

\begin{remark}
	It is possible to construct examples to show that not all the degree-$2$ CM extensions of $K_T^0$ are realized as the CM endomorphism fields of some polarized irreducible Hodge structures of K3 type $T' = \ell'^\perp$. Moreover, the same degree-$2$ CM extension of $K_T^0$ may occur for infinitely many different $T'=\ell'$, even for an Euclidean dense subset of $z_1 \in \QQ^+$ (again, $\ell'$ is chosen orthogonal to $z_1 \in \QQ^+$).
\end{remark}

Endow $K_T^\times$ with the topology induced as a subspace of $\C^\times$ by the fixed embedding $K_T \hookrightarrow \C$. Notice that the proof of Proposition \ref{proposition.altitude} also puts in evidence the continuity of the action of $K_T^\times$ on $\QQ$. We summarize all the results in the following corollary.

\begin{corollary}\label{corollary.action1}
	Assume that $T$ has CM. Then:
	\begin{itemize}
		\item the topological group $K_T^\times$ acts freely, transitively and continuously on $\QQ$. Having fixed an element $z_1 \in \QQ$, this action induces a homeomorphism between $K_T^\times$ and $\QQ$, and this homeomorphism induces a homeomorphism between $\C^\times$ and $\U$, and between $\C$ and $\U \cup \{x\}$;
		\item the subgroup of $K_T^\times$ given by the elements of norm one acts on $\QQ$ by rotation along the $\ell$-axis;
		\item all points in $\QQ$ at the same altitude of $z_1$ are obtained from $z_1$ by acting with an element of $K_T^\times$ of norm one. Besides, given a point $z_1 \in \QQ$, there exist countably many points at the same altitude, and they are dense in the corresponding circle of $\U$ at that altitude;
		\item if $z_1 \in \QQ^+$, all the other points of Picard jump at the same altitude are in $\QQ^+$, and the Hodge structures $\ell'^\perp$ and $\ell''^\perp$ corresponding to these points have the same CM;
		\item finally, $\QQ^+$ is dense in $\U$.
	\end{itemize}
\end{corollary}

\begin{proof}
	The only statements left to prove are the ones concerning the density of points of $\QQ$ at the same altitude and the density of points of $\QQ^+$. The first assertion follows from the fact that, for a CM field $E$, the set of points in $E \cap S^1$ is dense in $S^1$, according to Corollary \ref{corollary.densityCM}. For the second assertion: let $\gamma \in T$ such that $(\gamma.\gamma)>0$, so that $(\gamma+\ell.\gamma+\ell)=(\gamma.\gamma)+d>0$ as well. Set $\ell'=\gamma+\ell$. Consider an element in $\Q^\times \cdot (S^1 \cap K_T) \subseteq K_T^\times$, say $A=\lambda \alpha$, for $\lambda \in \Q^\times$ and $\alpha \in K_T^\times$ satisfying $\abs{\alpha}=1$. Then $\ell''=A(\gamma)+\ell$ satisfies
	\[
	(\ell''.\ell'')=(A(\gamma).A(\gamma))+d = \lambda^2 (\gamma.\gamma) + d >0.
	\]
	The density of $\QQ^+$ in $\U$ follows, therefore, from the density of $\Q^\times \cdot (S^1 \cap K_T)$ in $\C^\times$ ($S^1 \cap K_T$ is dense in $S^1$ thanks to Proposition \ref{proposition.density}).
\end{proof}

\subsection{On the equator}\label{section.actionequator}

Now we focus on the equator $S^1_\ell$. In the case $\PP^1_\ell \setminus S^1_\ell$ we had at our disposal a clean way to choose one of the two points orthogonal to a certain $\ell' \in T \oplus \Q \ell$: picking the one in the upper-half sphere (subject to the condition $\abs{a}>\abs{b}$). Here this choice is not so clear; therefore, we will consider pairs of antipodal points on $S^1_\ell$. Define $\RR$ to be the set of points of Picard jump on the equator modulo the relation $\{ \pm \}$. Pick an element $\ell'=\gamma_1 \in T$, $\gamma_1 \neq 0$, and consider $z_1=[\sigma'=a \sigma + b \overline{\sigma} + \ell]$ on $S^1_\ell$ which is orthogonal to $\gamma_1$. By our choice of $\RR$, here we mean at the same time $z_1$ and $\overline{z_1}$. For $A \in K_T^\times$, we define $z_2 = A * z_1$ to be the (pair of) point orthogonal to $\ell''=\gamma_2=A(\gamma_1)$. This defines an action of $K_T^\times$ on $\RR$. This action is no longer free: indeed, as explained in the proof of Proposition \ref{proposition.picardjumpeq}, two non-zero elements $\gamma, \delta$ in $T$ are orthogonal to the same point if there exists a $\beta \in (K_T^0)^\times= K_T^\times \cap \R$ such that $\beta(\gamma)=\delta$. Then, $(K_T^0)^\times$ is contained in the kernel of this action, or equivalently in the stabilizer of each point. This induces an action of the quotient group $K_T^\times / (K_T^0)^\times$ on $\RR$. If $T$ has CM, the action of the quotient is free, i.e.\ $(K_T^0)^\times$ is exactly the stabilizer of each point: this follows from the fact that $\rho_z = r/2$ for points of Picard jump on the equator (Proposition \ref{proposition.picardjumpeq}). However, it will be clear later that the action of $K_T^\times/(K_T^0)^\times$ is free even if $T$ is not of CM type. Notice that the quotient group is trivial if $T$ is of totally real type. Again, if $T$ has CM then the action is transitive (as $\dim_{K_T} T = 1$).

\medskip

We have an explicit way to compute $z_2 = A * z_1$. Pick $\sigma$ satisfying $(\sigma.\ell')=1$. If $z_1=[\sigma'=a \sigma + b \overline{\sigma} +\ell]$, (\ref{equation.lin}) gives $a+b=0$, i.e.\ $b=-a$. On the other hand, we rewrite (\ref{equation.quadr}) as $-2a^2(\sigma.\overline{\sigma})+d=0$, which forces $a \in \R$. Hence,
\[
\sigma'=a\sigma - a \overline{\sigma}+\ell.
\]
Assume that $A = R e^{i\theta}$; then $z_2=[\sigma'']$, where
\[
\sigma''=ae^{i\theta}\sigma -ae^{-i\theta} \overline{\sigma} +\ell.
\]
Indeed $z_2 \in \PP^1_\ell$ holds since (\ref{equation.quadr}) is satisfied, and
\begin{align*}
(\sigma''.A(\gamma_1))&=(ae^{i\theta}\sigma.A(\gamma_1))+(-ae^{-i\theta}\overline{\sigma}.A(\gamma_1))\\
&=ae^{i\theta}\big(\overline{A}(\sigma).\gamma_1\big)-ae^{-i\theta}\big(\overline{A}(\overline{\sigma}).\gamma_1\big)\\
&=ae^{i\theta}\overline{A} -ae^{-i\theta} A=0.
\end{align*}
Looking at the isomorphism $\PP^1_\ell \simeq S^2_\ell$ defined in Lemma \ref{lemma.parametrization}, we see that $A=R e^{i\theta}$ acts on $\RR$ by a rotation of angle $-\theta$. Note that if $A \in (K_T^0)^\times$ then $\theta=0$ or $\theta= \pi$, and $z_2=z_1$ (in $\RR$), as already mentioned.

\begin{center}
	\includegraphics[width=0.3\textwidth]{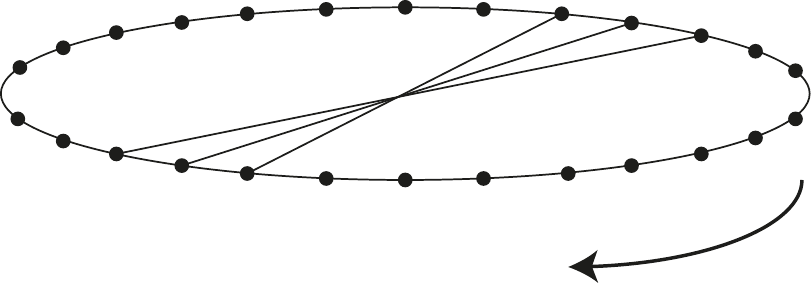}
\end{center}

We add a topological flavour to this discussion. Endow $K_T^\times$ with the topology induced by the topology of $\C^\times$ under the fixed embedding $K_T \hookrightarrow \C$, and $(K_T^0)^\times$ with the subspace topology. Then the topological quotient group $K_T^\times / (K_T^0)^\times$ is a subspace of the topological quotient group $\PP^1(\R)=\C^\times/\R^\times$. Moreover, $\RR$ is a subspace of $S^1_\ell / \{\pm\}$ and the action of $K_T^\times / (K_T^0)^\times$ on it is given by rotation.

\begin{remark}
	If $K_T$ is a CM field, then $K_T^\times / (K_T^0)^\times$ is dense in $\PP^1(\R)$. To deduce this, it is enough to show that the elements of norm one inside $K_T$ are dense in the circle: this is the content of Corollary \ref{corollary.densityCM}.
\end{remark}

\begin{remark}\label{remark.R+}
	We may define the equivalent of $\QQ^+$ also in this setting: let $\RR^+$ be the set of (pair of) points $z \in S^1_\ell \subseteq \PP^1_\ell$ that are orthogonal to an element of $T$ of positive self-intersection. If $\gamma_1 \in T$ is of positive self intersection, denote by $z_1$ (one of) its orthogonal. For an element $A \in K_T^\times$, $\abs{A}=1$ implies that $\gamma_2=A(\gamma_1)$ is of positive self-intersection as well. Then, by acting with elements of $S^1 \cap K_T^\times$ on $z_1$, we see that $\RR^+$ is dense in $S^1_\ell/\{\pm\}$ if $K_T$ is a CM field.
\end{remark}

In retrospect, we have proven the following result.

\begin{proposition}\label{proposition.action2}
	Assume that $T$ has CM. Then:
	\begin{itemize}
		\item the topological group $K_T^\times / (K_T^0)^\times$ acts freely, transitively and continuously by rotations on the topological space $\RR$;
		\item having fixed an element $z_1 \in \RR$, this action induces a homeomorphism between $K_T^\times / (K_T^0)^\times$ and $\RR$, and this homeomorphism passes to the topological completions, inducing a homeomorphism between $\PP^1(\R)$ and $S^1_\ell/\{\pm\}$;
		\item $\RR$ is dense in $S^1_\ell/\{\pm\}$, or equivalently: the set of points of Picard jump of the equator is (countable and) dense in the equator;
		\item finally, $\RR^+$ is countable and dense in $S^1_\ell/\{\pm\}$.
	\end{itemize}
\end{proposition}

\section{Geometric interpretation}\label{section.geometry}

Suppose that $X$ is a complex projective K3 surface, and let $\ell=c_1(L)$ be the first Chern class of an ample line bundle. We call period field and endomorphism field of $X$ the period field and the endomorphism field, respectively, of the transcendental lattice $T(X) \subseteq H^2(X,\Q)$, which is a polarized irreducible Hodge structure of K3 type. We say that a complex projective K3 surface $X$ has CM if $T(X)$ has CM, that is: its dimension as a $K_{T(X)}$-vector space is $1$.

\medskip

We can construct the twistor space $\XX \to \PP^1_\C$ associated with $(X,L)$, where we put in evidence the considered ample line bundle. The fibre over a point $\zeta \in \PP^1_\C$ is a K3 surface of complex structure $x_1 I + x_2 J + x_3 K$, if $\zeta$ and $(x_1,x_2,x_3)$ correspond one to the other via stereographic projection. One should be careful, for the presence of two spheres: $S^2_\ell\simeq \PP^1_\ell$, an element of which corresponds to a class $z =[\sigma'] \in \PP^1_\ell$; $S^2\simeq \PP^1_\C$, parametrizing the complex structures $x_1 I + x_2 J + x_3 K$ of the fibres of the geometric twistor space. We want to underline the relation existing between these spheres. We may give explicitly an isomorphism $\PP^1_\ell \simeq \PP^1_\C$ that respects the following property: the Hodge structure determined by $z=[\sigma'] \in \PP^1_\ell$ on $T \oplus \Q \ell$ corresponds to the Hodge structure determined by $\sigma_\zeta$ on the same vector space, for $\zeta \in \PP^1_\C$ corresponding to $z \in \PP^1_\ell$. Equivalently, one may ask $\sigma'$ and $\sigma_\zeta$ to differ only by a a complex scalar. The isomorphism $\PP^1_\C \simeq \PP^1_\ell$ is defined by
\[
\zeta \mapsto \big[\sigma_\zeta = \sigma - \zeta^2 \overline{\sigma} + 2 \zeta \ell\big] = \left[ \frac{1}{2\zeta} \sigma - \frac{\zeta}{2}\overline{\sigma} + \ell \right]
\]
(sending $\infty$ to $[\overline{\sigma}]$). This follows from the explicit form of $\sigma_\zeta$ given in \cite[Section 3.F]{HKLR}. It is remarkable that $\sigma_\zeta$ is a linear combination of $\sigma, \overline{\sigma}$ and $\ell$ only. Composing these isomorphisms, we get
\[
S^2_\ell \simeq \PP^1_\ell \simeq \PP^1_\C \simeq S^2.
\]
We claim that this composition is nothing but the permutation $(x,y,z)\mapsto (x_1,x_2,x_3)=(z,x,y)$ on the sphere. For, a point $\zeta \in \PP^1_\C$ is sent, towards the left, to $\left[ \frac{1}{2\zeta} \sigma - \frac{\zeta}{2}\overline{\sigma} + \ell \right] \in \PP^1_\ell$. Set $a=1/(2\zeta)$, $b=-\zeta/2$. Then
\[
b-\overline{a} = - \frac{\zeta}{2} - \frac{1}{2\overline{\zeta}} =-\frac{1}{2\overline{\zeta}} (\zeta \overline{\zeta}+1) =\frac{\zeta\overline{\zeta}+1}{4\zeta\overline{\zeta}} \cdot (-2\zeta)
\]
and
\[
a\overline{a} - b \overline{b} = \frac{1}{4\zeta \overline{\zeta}} - \frac{\zeta \overline{\zeta}}{4} = \frac{\zeta \overline{\zeta} +1}{4\zeta \overline{\zeta}} \cdot (1-\zeta \overline{\zeta}).
\]
Therefore the vector
\[
v(a,b)= \left( \Re (b-\overline{a}), \Im (b-\overline{a}), \big(a\overline{a}-b\overline{b}\big)\frac{(\sigma.\overline{\sigma})}{2d}\right),
\]
that determines the point in $S^2_\ell$ (Lemma \ref{lemma.parametrization}), is positively aligned to the vector
\[
\left(-2 \Re(\zeta), -2 \Im(\zeta), 1- \zeta \overline{\zeta}\right),
\]
where we use that $(\sigma.\overline{\sigma})=2(\ell.\ell)=2d$, as already noticed. However, the image of $\zeta$ in $S^2$ via the stereographic projection is positively aligned with
\[
\left(1- \zeta \overline{\zeta}, -2 \Re(\zeta), -2 \Im(\zeta)\right),
\]
and this proves the claim.

\medskip

Therefore, we may talk about equator and points at the same altitude for $S^2_\ell$ and $S^2$ interchangeably (even if the altitude of $S^2_\ell$ corresponds, in fact, to the first coordinate of $S^2$). Points at the same altitude correspond, therefore, to complex structures $x_1 I + x_2 J + x_3 K$ on $X$ with the same coefficient $x_1$. The equator correspond to complex structures $x_2 J + x_3 K$, for which the $I$-part is missing.

\begin{remark}\label{remark.differentpicardnumbers}
	For a point $\zeta \in \PP^1_\C$, let $z \in \PP^1_\ell$ its correspondent under the isomorphism $\PP^1_\ell \simeq \PP^1_\C$. We have the equality
	\[
	\rho_z + \rho(X) -1=\rho(\XX_\zeta).
	\]
	To prove this, look at the Néron-Severi group of the K3 surface $\XX_\zeta$. For an element $\gamma \in H^2(\XX_\zeta,\Q) = H^2(X,\Q)$, being in $\NS(\XX_\zeta)$ is equivalent to being orthogonal to the $(2,0)$-form $\sigma_\zeta=\sigma - \zeta^2 \overline{\sigma} + 2 \zeta \ell$. Of course the orthogonal complement of the class $\ell$ in the original Néron-Severi group, $\ell^\perp \subseteq \NS(X)$, is always orthogonal to $\sigma_\zeta$ (and this corresponds to the addend $\rho(X)-1$). The orthogonal complement of this last space in $H^2(X,\Q)$ is $T\oplus \Q \ell$, and is in direct sum with it. Lying in $T \oplus \Q \ell$ and, at the same time, being orthogonal to $\sigma_\zeta$ means belonging to the space
	\[
	P_z^\perp \cap (T \oplus \Q \ell),
	\]
	whose dimension is, by definition, $\rho_z$. Therefore, the formula above holds.
\end{remark}

In the following, when talking about the twistor space $\XX \to \PP^1_\C$ associated with a projective complex K3 surface $X$, we imply the choice of $\ell=c_1(L)$ as K\"ahler class for the construction of $\XX$. The analogous of Proposition \ref{proposition.jump}, proven by Huybrechts, is the following result.

\begin{corollary}
	Consider the twistor space $\XX \to \PP^1_\C$ associated with a projective complex K3 surface $X$. If $\rho(\XX_\zeta) > \rho(X)$, then $\zeta$ is contained in the equator $S^1 \subseteq S^2$ (that is: $\zeta$ is mapped to the equator $S^1 \subseteq S^2$ under stereographic projection $\PP^1_\C \simeq S^2$).
\end{corollary}

Proposition \ref{proposition.picardjumpeq} and Proposition \ref{proposition.action2} yield:

\begin{corollary}\label{theorem.geometricpicardjump}
	Consider the twistor space $\XX \to \PP^1_\C$ associated with a projective complex K3 surface $X$ with complex multiplication. If $\zeta$ is a point of Picard jump on the equator, then
	\[
	\rho(\XX_\zeta) =  10 + \frac{\rho(X)}{2}.
	\]
	Moreover, the Noether-Lefschetz locus of the equator is dense in the equator. More precisely, the locus of points on the equator whose fibres are algebraic K3 surfaces is dense in the equator.
\end{corollary}

\begin{center}
	\vcenteredhbox{\begin{tabular}{cc}
			\toprule
			&
			\begin{tabular}{c}
				Admissible values \\
				of $\rho(\XX_\zeta)$, CM case \\
			\end{tabular}
			\\
			\toprule
			\begin{tabular}{c}
				Outside the \\
				equator \\
			\end{tabular}
			& $\rho(X)-1, \rho(X)$ \\
			\midrule
			\begin{tabular}{c}
				On the \\
				equator \\
			\end{tabular}
			& $\rho(X)-1, 10 + \frac{\rho(X)}{2}$ \\
			\bottomrule
	\end{tabular}}
	\qquad \qquad
	\vcenteredhbox{\includegraphics[width=0.3\textwidth]{sphere}}
\end{center}

\begin{proof}
	Only the equation is left to discuss. If $r = \dim_\Q T$, then $r + \rho(X) = 22$ and $\rho_z = r/2 = (22-\rho(X))/2$, if $z$ corresponds to $\zeta$ via the isomorphism $\PP^1_\ell \simeq \PP^1_\C$. Therefore, the formula above follows from Remark \ref{remark.differentpicardnumbers}.
\end{proof}

\begin{remark}
	The only case where no points of excessive jump appear is under the assumption of maximal Picard number, i.e.\ $\rho(X)=20$. This answers \cite[Remark 5.2]{Huy} in the CM case. Also, it is a generalization of \cite[Remark 5.5]{Huy}: not only if $\rho(X)<20$ there is no fibre such that $\rho(\XX_\zeta)=20$, but the set of admissible values of $\rho(\XX_\zeta)$ is also very constrained.
\end{remark}

\begin{remark}
	A K3 surface $X$ is projective, or equivalently algebraic, if and only if there exists a line bundle $L$ with $L^2>0$, where $L^2$ denotes the self-intersection of this line bundle. For a proof, see \cite[Theorem IV.6.2]{BHPV}. For twistor fibres, we see how the geometric property ``being algebraic'' agrees with the algebraic requirement $(\ell'.\ell') > 0$.
\end{remark}

\begin{remark}
	If $\rho(X)=20$, each fibre $\XX_\zeta$ corresponding to a point of Picard jump (both on and outside the equator) has Picard number $20$ and, therefore, it is automatically algebraic and of CM type. However, when $\rho(X)<20$, thanks to Remark \ref{remark.QneqQ+}, there are non-algebraic fibres corresponding to points of Picard jump outside the equator.
\end{remark}

Corollary \ref{corollary.action1} yields the following results.

\begin{proposition}\label{proposition.geom1}
	Consider the twistor space $\XX \to \PP^1_\C$ associated with a projective complex K3 surface $X$ with complex multiplication. Then the locus of $\zeta \in \PP^1_\C$ such that $\XX_\zeta$ is algebraic is dense (for the classical topology) in $\PP^1_\C$.
\end{proposition}

\begin{proof}
	The density of the locus follows from the density of $\QQ^+$ in $\U$.
\end{proof}

\begin{theorem}\label{theorem.sameCM}
	Consider the twistor space $\XX \to \PP^1_\C$ associated with a projective complex K3 surface $X$ with complex multiplication. Assume that $\zeta_1, \zeta_2 \in \PP^1_\C$ are two points of Picard jump at the same altitude and not on the equator. Then $\XX_{\zeta_1}$ is algebraic if and only if $\XX_{\zeta_2}$ is such. If so, then the CM endomorphism fields of these K3 surfaces coincide. Moreover, the set of points of Picard jump at the same altitude of $\zeta_1$ (and $\zeta_2$) is countable and dense in the circle at that altitude.
\end{theorem}

\begin{center}
	\includegraphics[width=0.3\textwidth]{noarrow}
\end{center}

\begin{remark}
	For $\zeta_1, \zeta_2 \in \PP^1_\C$, being at the same altitude and outside the equator means corresponding to complex structures on $X$ having the same non-zero $I$-component.
\end{remark}

\newpage

\printbibliography

\end{document}